\documentclass[11pt]{article}
\usepackage{amsthm, amsmath, amssymb, amsfonts, url, booktabs, tikz, setspace, fancyhdr, amsbsy}
\usepackage[margin = 1in]{geometry}
\usepackage{esint}

\newtheorem{theorem}{Theorem}[section]
\newtheorem{proposition}[theorem]{Proposition}
\newtheorem{lemma}[theorem]{Lemma}
\newtheorem{corollary}[theorem]{Corollary}

\theoremstyle{definition}

\theoremstyle{remark}

\newcommand{\norm}[1]{\left\lVert#1\right\rVert}

\newcommand{\R}{\mathbb{R}}

\newcommand{\defeq}{\mathrel{\mathop:}=}
\newcommand{\eqdef}{\mathrel{\mathop=}:}

\newcommand{\uu}{\boldsymbol{u}}
\newcommand{\DD}{\boldsymbol{D}}

\newcommand{\vv}{\boldsymbol{v}}
\newcommand{\SSS}{\boldsymbol{S}}
\newcommand{\q}{\boldsymbol{q}}
\newcommand{\ww}{\boldsymbol{w}}

\newcommand{\dx}{\,\mathrm{d}x}

\newcommand{\dt}{\,\mathrm{d}t}
\newcommand{\aaa}{\boldsymbol{a}}

\numberwithin{equation}{section}

\begin{document}

\title{Existence of global weak solutions for unsteady motions of incompressible chemically reacting generalized Newtonian fluids}

\author{Seungchan Ko\thanks{Data and Information Technology Center, Samsung Electronics, 18488, Hwaseong-si, Republic of Korea. Email: \tt{ksm0385@gmail.com}} }

\date{~}
\maketitle
~\vspace{-1.5cm}

\begin{abstract}
We study a system of nonlinear partial differential equations describing the unsteady  motions of incompressible chemically reacting non-Newtonian fluids. The system under consideration consists of the generalized Navier--Stokes equations with a power-law type stress-strain relation, where the power-law index depends on the concentration of a chemical, coupled to a convection-diffusion equation for the concentration. This system of equations arises in the rheology of the synovial fluid found in the cavities of synovial joints. We prove the existence of global weak solutions of the non-stationary model by using the Galerkin method combined with generalized monotone operator theory and parabolic De Giorgi--Nash--Moser theory.
As the governing equations involve a nonlinearity with a variable power-law index, our proofs exploit the framework of generalized Sobolev spaces with variable integrability exponent.
\end{abstract}

\smallskip


\smallskip


\begin{section}{Introduction}
We are interested in developing an existence theory for a system of partial differential equations (PDEs), which consists of the unsteady incompressible generalized Navier--Stokes equations, wherein the viscosity is a polynomial function of the shear-rate with the power-law index depending on the concentration, and a convection-diffusion equation for the concentration. We would like to know whether the weak solutions to this system exist in the domain $Q_T\defeq\Omega\times(0,T)$, where $\Omega\subset\R^d$ with $d\geq2$ is a bounded open Lipschitz domain and $(0,T)$ denotes the time interval of interest. Namely, we study the following system of PDEs:
\begin{alignat}{2}
\text{div}\,{\uu}&=0\qquad &&\mbox{in $Q_T$},\label{eq1}\\
\partial_t\uu+\text{div}\,(\uu\otimes \uu)-\text{div}\,\SSS(c,\DD\uu)&=-\nabla \pi+\boldsymbol{f}\qquad &&\mbox{in $Q_T$},\label{eq2}\\
\partial_tc+\text{div}\,(c\uu)-\text{div}\,\q_c(c,\nabla c,\DD\uu)&=0\qquad &&\mbox{in $Q_T$},\label{eq3}
\end{alignat}
where $\uu:Q_T\rightarrow\R^d$, $\pi:Q_T\rightarrow\R$, $c:Q_T\rightarrow\R_{\geq0}$ are the velocity field, pressure and concentration respectively. In the present context, $\boldsymbol{f}:Q_T\rightarrow\R^d$ represents a given density of the bulk force, $\DD\uu$ denotes the symmetric velocity gradient, i.e., $\DD\uu=\frac{1}{2}(\nabla \uu+(\nabla \uu)^T)$, and $\SSS(c,\DD\uu)$ and $\q_c(c,\nabla c,\DD\uu)$ are the extra stress tensor of the Cauchy stress tensor and the diffusion flux respectively. In this setting, for given functions $\uu_0$ and $c_0$ defined in $\Omega$, we prescribe the following initial conditions
 \begin{equation}\label{IC}
 \uu(x,0)=\uu_0(x)\qquad{\rm{and}}\qquad c(x,0)=c_0(x)\qquad{\rm{in}}\,\,\Omega,
 \end{equation}
where $0\leq c_0\leq \tilde{c}_0$ for some positive constant $\tilde{c}_0$. 

Furthermore, we set $S_T\defeq \partial\Omega\times (0,T)$ and $\Gamma_T\defeq S_T\cup\{(x,t):x\in\overline{\Omega}, t=0\}$, and we prescribe the homogeneous boundary conditions
\begin{equation}
\uu(x,t)=\boldsymbol{0}\qquad{\rm{and}}\qquad c(x,t)=0\qquad{\rm{on}}\,\,S_T.
\end{equation}

Finally, as described above, we assume that the extra stress tensor $\SSS:\R_{\geq 0}\times\R^{d\times
d}_{\rm sym}\rightarrow\R^{d\times d}_{\rm sym}$ is a continuous mapping
with the following non-standard growth, strict monotonicity and coercivity
conditions: there exist positive constants $C_1$, $C_2$ and $C_3$ such that
\begin{equation}\label{S1}
|\SSS(c,\boldsymbol{B})|\leq C_1(|\boldsymbol{B}|^{p(c)-1}+1),
\end{equation}
\begin{equation}\label{S2}
(\SSS(c,\boldsymbol{B_1})-\SSS(c,\boldsymbol{B_2})):(\boldsymbol{B_1}-\boldsymbol{B_2})>0\,\,\,\text{for}\,\,\boldsymbol{B_1}\neq
\boldsymbol{B_2},
\end{equation}
\begin{equation}\label{S3}
\SSS(c,\boldsymbol{B})\cdot \boldsymbol{B}\geq C_2(|\boldsymbol{B}|^{p(c)}+|\SSS|^{p'(c)})-C_3,
\end{equation}
where $p:\R_{\geq0}\rightarrow\R_{\geq0}$ is a H\"older continuous function such that $1<p^-\leq p(c)\leq p^+<\infty$ and $p'(c)$ is defined as $\frac{p(c)}{p(c)-1}$.
Furthermore, we assume that the diffusion flux vector
$\q_c(c,\boldsymbol{g},\boldsymbol{B}):\R_{\geq 0}\times\R^d\times\R^{d\times d}_{\rm{sym}}\rightarrow\R^d$ is a
continuous mapping, which is linear with respect to $\boldsymbol{g}$, and satisfies the following inequalities: there exist positive constants $C_4$ and $C_5$ such that
\begin{align}
|\q_c(c,\boldsymbol{g},\boldsymbol{B})|&\leq C_4|\boldsymbol{g}|,\label{q1}\\
\q_c(c,\boldsymbol{g},\boldsymbol{B})\cdot \boldsymbol{g}&\geq C_5|\boldsymbol{g}|^2.\label{q2}
\end{align}

The prototypical examples satisfying all the assumptions \eqref{S1}--\eqref{q2} are of the following forms:
\begin{equation}\label{proto}
\SSS(c,\DD\uu)=\nu(c,|\DD\uu|)\DD\uu,\qquad \q_c(c,\nabla c,\DD\uu)=\boldsymbol{K}(c,|\DD\uu|)\nabla c,
\end{equation}
where the viscosity $\nu(c,|\DD\uu|)$ depends on the shear-rate and on the concentration in the following fashion: 
\begin{equation}\label{viscosityform}
\nu(c,|\DD\uu|)\sim\nu_0(1+|\DD\uu|^2)^{\frac{p(c)-2}{2}},
\end{equation}
where $\nu_0$ is a positive constant. 

This model is a generalization of the classical power-law model, and the mathematical study of this type of fluid flow model is closely related to our problem. The rigorous mathematical study of power-law fluids began in the late nineteen sixties with the pioneering works of Lions and Ladyzhenskaya in \cite{Lions} and \cite{lady2}, \cite{lady1}, respectively. In these papers, the authors developed the existence theory for weak solutions of non-stationary models for $p\geq\frac{3d+2}{d+2}$ with the aid of monotone operator theory. Since then, there has been significant progress in the mathematical theory of non-Newtonian fluids. To obtain the compactness of the convective term, some useful approximation techniques were used such as $L^{\infty}$ and Lipschitz truncations, and thus the existence of weak solutions with a wide range of the (constant) power-law index $p>\frac{2d}{d+2}$ was established; see \cite{FMS1997, R1997, FMS2003, DMS2008} in the case of steady, and \cite{MNR2001, W2007, DRW2010, BDS2013} for unsteady flows.

Interestingly, it was discovered through laboratory experiments that in a number of non-Newtonian fluid flow models, the power-law index $p$ is not a fixed constant but is, rather, variable; see, e.g., \cite{RRS1996} for a model where the power-law index is of the form $p(\cdot)\defeq p(|\boldsymbol{E}(x)|^2)\sim p(x)$, $\boldsymbol{E}$ being a given electric field. Such electro-rheological models were investigated in \cite{R2000, R2004}, where the existence theory was developed. Recently, the Lipschitz truncation method was generalized to variable-exponent spaces in \cite{DMS2008}, where the existence of weak solutions was shown provided that $p(\cdot)$ is sufficiently regular in some sense, $1<p^-\leq p(x)\leq p^+<\infty$ and $p^->\frac{2d}{d+2}$.

The rigorous mathematical study of a PDE system, consisting of the Navier--Stokes equations of power-law type, with a concentration-dependent viscosity, coupled to a convection-diffusion equation, was initiated in \cite{BMR2008}. In that paper, however, the power-law index was a fixed number and the concentration was merely a scaling factor of the viscosity, i.e., $\nu(c,|\boldsymbol{D}|^2)\sim f(c)\tilde{\nu}(|\boldsymbol{D}|^2)$. However, for some biological fluids including the synovial fluid, it was shown that the concentration of a particular molecule is not just a scaling factor of the viscosity, but it influences the rate at which the fluid thins the shear. For instance, if the concentration is zero, the viscosity remains constant for different shear-rates. Additionally, for increasing concentration in the solution, the fluid exhibits higher apparent viscosity, and it thins the shear more remarkably. Therefore, a new model describing these fluids was proposed where the shear-thinning index itself is a function of the concentration. Such a model is able to describe the viscous properties of these biological fluids in a better way than the model used in \cite{BMR2008}. For the rheological background of this model, see \cite{exp, exp2}.

In \cite{BP2013, BP2014}, the authors started studying this new model. The existence of a weak solution to the elliptic problem was proved in the case when the variable exponent $p(x)$ is bounded below by $\frac{3d}{d+2}$ and $\frac{d}{2}$ respectively. From the viewpoint of computational mathematics, in \cite{KPS2017}, a construction of a conforming finite element approximation of the stationary model was considered and the convergence analysis of the numerical method was developed in the case of variable-exponent spaces in a two dimensional domain. Subsequently, the analysis was extended to three space dimensions in \cite{KS2017} by using a slightly different numerical scheme. 

To the best of our knowledge, there is no existence result for this new model in the case of unsteady fluid flows. In this paper, we shall develop the existence theory for the non-stationary model with the help of a generalized Galerkin method, motivated by \cite{BMR2009}, together with the use of monotone operator theory in variable-exponent spaces and the parabolic De Giorgi--Nash--Moser regularity theory. As a consequence, we can also obtain the existence result for a class of unsteady electro-rheological fluid flow models as a special case, which is completely new.

\end{section}

\begin{section}{Main result}
In this section, we introduce our main theorem together with some preliminaries required for the precise statement of the theorem. We shall write $z=(x,t)\in\R^{d+1}$, where $x\in\R^d$ denotes the spatial variable and $t\in\R_{\geq0}$ is the time variable. Let $\mathcal{P}$ be a set of all measurable functions $p:Q_T\rightarrow[1,\infty]$; we call a function $p\in\mathcal{P}(Q_T)$ a variable exponent. Then we define $p^-\defeq$ ess $\inf_{z\in Q_T}p(z)$, $p^+\defeq$ ess $\sup_{z\in Q_T}p(z)$. Throughout the paper, we only deal with the case
\begin{equation}\label{pcondition}
1<p^-\leq p^+<\infty.
\end{equation}

Since we are considering the concentration-dependent power-law index, we need to use generalized Lebesgue and Sobolev spaces with variable exponents. Precisely, we first introduce generalized Lebesgue spaces, which are equipped with the corresponding Luxembourg norms
\begin{align*}
L^{p(\cdot)}(Q_T)&\defeq \left\{u\in L^1_{\rm{loc}}(Q_T):\int_{Q_T}|u(x,t)|^{p(x,t)}\dx\dt<\infty\right\},\\
\norm{u}_{L^{p(\cdot)}(Q_T)}&\defeq \inf\left\{\lambda>0:\int_{Q_T}\bigg|\frac{u(x,t)}{\lambda}\bigg|^{p(x,t)}\dx\dt\leq1\right\}.
\end{align*}
Some basic inequalities, which are well-known for classical Lebesgue spaces, also hold for variable-exponent Lebesgue spaces. For example, if $p$, $q$, $s\in\mathcal{P}(Q_T)$ are such that $\frac{1}{s(z)}=\frac{1}{p(z)}+\frac{1}{q(z)}$ for all $z\in Q_T$, then for any $f\in L^{p(\cdot)}(Q_T)$ and $g\in L^{q(\cdot)}(Q_T)$,  we have
\begin{itemize}
\item (H\"older's inequality) $\qquad\|fg\|_{L^{s(\cdot)}(Q_T)}\leq2\|f\|_{L^{p(\cdot)}(Q_T)}\|g\|_{L^{q(\cdot)}(Q_T)}$,
\item (Young's inequality) $\qquad\int_{Q_T}|fg|^{s(x,t)}\dx\dt\leq\int_{Q_T}|f|^{p(x,t)}\dx\dt+\int_{Q_T}|g|^{q(x,t)}\dx\dt,$
\end{itemize}
which are not difficult to verify.

Next we define the following anisotropic Sobolev space with variable exponent:
\begin{align*}
W_{p(\cdot)}(Q_T)&\defeq \left\{u\in L^1(0,T;W^{1,1}_{0,\,{\rm{div}}}(\Omega)^d):u\in L^2(Q_T)^d\,\,\,{\rm{and}}\,\,\,|\nabla u|\in L^{p(\cdot)}(Q_T)\right\},\\
\norm{u}_{W_{p(\cdot)}(Q_T)}&\defeq \|u\|_{L^2(Q_T)}+\|\nabla u\|_{L^{p(\cdot)}(Q_T)}.
\end{align*}

It is easy to check that all of the above spaces are Banach spaces, and thanks to \eqref{pcondition}, they are all separable and reflexive; see \cite{DHHR2011}. We also define the dual space $(L^{p(\cdot)}(Q_T))^*=L^{p'(\cdot)}(Q_T)$ where the dual variable exponent $p'\in\mathcal{P}(Q_T)$ is defined by $\frac{1}{p(z)}+\frac{1}{p'(z)}=1.$

Furthermore, in order to study the temporal regularity of functions, we introduce the following parabolic spaces:
\begin{align*}
C([0,T];L^2(\Omega))&\defeq\{u\in L^{\infty}(0,T;L^2(\Omega)):u:[0,T]\rightarrow L^2(\Omega)\,\,\,{\rm{is}}\,\,\,{\rm{continuous}}\},\\
C_{\rm{w}}([0,T];L^2(\Omega))&\defeq\{u\in L^{\infty}(0,T;L^2(\Omega)):u:[0,T]\rightarrow L^2(\Omega)\,\,\,{\rm{is}}\,\,\,{\rm{weakly}}\,\,\,{\rm{continuous}}\},
\end{align*}
where both spaces are equipped with the $L^{\infty}(0,T;L^2(\Omega))$ norm.

Finally, we need to introduce the parabolic H\"older space, which has an important role in our analysis. We consider the parabolic metric $d_p$ defined by
\[d_p(z_1,z_2)\defeq |x_1-x_2|+|t_1-t_2|^{\frac{1}{2}},\]
where $z_1=(x_1,t_1)$, $z_2=(x_2,t_2)\in\R^{d+1}$. We then define the parabolic H\"older space using the above parabolic distance: for some $\alpha\in(0,1)$, 
\[C^{\alpha,\alpha/2}(Q_T)\defeq\{f\in C(Q_T):\exists C>0\,\,\,{\rm{such}}\,\,\,{\rm{that}}\,\,\,|f(z_1)-f(z_2)|\leq Cd_p(z_1,z_2)^{\alpha}\,\,\,{\rm{for}}\,\,\,{\rm{any}}\,\,\,z_1,z_2\in Q_T\},\]
with the norm
\[\|f\|_{C^{\alpha,\alpha/2}(Q_T)}\defeq\|f\|_{L^{\infty}(Q_T)}+\sup_{z_1,z_2\in Q_T,\,z_1\neq z_2,}\frac{|f(z_1)-f(z_2)|}{d_p(z_1,z_2)^{\alpha}}.\]

Then, our main result concerning the existence of weak solutions is as follows:
\begin{theorem}\label{mainthm}
Let $\Omega\subset\R^d$ with $d\geq2$ be a bounded open Lipschitz domain, and assume that $\boldsymbol{f}\in L^2(Q_T)^d$, $\uu_0\in L^2(\Omega)^d$ and $c_0\in \{f\in C^{\alpha_0}(\overline{\Omega})\,\,\,{\rm{for}}\,\,\,{\rm{some}}\,\,\,\alpha_0\in(0,1):f=0\,\,\,{\rm{on}}\,\,\,\partial\Omega\}$.  If $p$ is a H\"older continuous function with $p^->\frac{d+2}{2}$, then there exists a weak solution pair $(\uu,c)$ to the system of equations \eqref{eq1}--\eqref{eq3} and some $\alpha\in(0,1)$ such that
\begin{align*}
\uu&\in C([0,T];L^2(\Omega)^d)\cap L^{p^-}(0,T;W^{1,p^-}_{0,\,{\rm{div}}}(\Omega)^d)\cap W_{p(c)}(Q_T),\\
c&\in C([0,T];L^2(\Omega))\cap L^{2}(0,T;W^{1,2}_0(\Omega))\cap C^{\alpha,\alpha/2}(\overline{Q}_T),
\end{align*}
satisfying
\[\int_{Q_T}\SSS(c,\DD\uu):\boldsymbol{D\psi}\dx\dt
=\int_{Q_T}\left(\boldsymbol{f}\cdot\boldsymbol{\psi}+(\uu\otimes\uu):\boldsymbol{D\psi}+\uu\cdot\partial_t\boldsymbol{\psi}\right)\dx\dt+\int_{\Omega}\uu_0(x)\cdot\boldsymbol{\psi}(x,0)\dx\]
for all $\boldsymbol{\psi}\in C^{\infty}_{0,{\rm{div}}}(\Omega\times [0,T))^d$ and
\[\int_{Q_T}\q_c(c,\nabla c,\DD\uu)\cdot\nabla\varphi\dx\dt=\int_{Q_T}\left(c\partial_t\varphi+c\uu\cdot\nabla\varphi\right)\dx\dt+\int_{\Omega}c_0(x)\varphi(x,0)\dx\]
for all $\varphi\in C^{\infty}_0(\Omega\times [0,T))$.
\end{theorem}

\end{section}

\begin{section}{Auxiliary tools and results}
In this section, we introduce the necessary technical tools, especially from function space theory, that we use throughout the paper. We begin with some embedding theorems. The following proposition concerns the compact embedding of parabolic H\"older spaces.
\begin{proposition}\label{Holdercpt}
Let $Q_T\subset\R^{d+1}$ be such that $\overline{Q}_T$ is compact in $\R^{d+1}$ and let $0<\alpha<\beta<1$. Then the inclusion map $\mathit{i}:C^{\beta,\beta/2}(\overline{Q}_T)\hookrightarrow C^{\alpha,\alpha/2}(\overline{Q}_T)$ is compact.
\end{proposition}
\begin{proof}
We first note that $\overline{Q}_T$ is also compact with respect to the topology induced by the parabolic distance $d_p(\cdot,\cdot)$. Let $\{f_n\}^{\infty}_{n=1}$ be a bounded sequence in $C^{\beta,\beta/2}(\overline{Q}_T)$. Then, trivially $\{f_n\}^{\infty}_{n=1}$ is a uniformly bounded and equicontinuous class of functions. By the Arzel\'a--Ascoli theorem, there exists a subsequence $\{f_n\}^{\infty}_{n=1}$ (not relabelled) such that $f_n\rightarrow f$ in $C(\overline{Q}_T)$ for some $f\in C^{\beta,\beta/2}(\overline{Q}_T)$. 

Let $g_n\defeq f-f_n\in C^{\beta,\beta/2}(\overline{Q}_T)$. Then $\|g_n\|_{C^{\beta,\beta/2}(\overline{Q}_T)}\leq C$ for some constant $C>0$ and $\|g\|_{C(\overline{Q}_T)}\rightarrow0$ as $n\rightarrow\infty$. To complete the proof, it suffices to show that 
\[[g_n]_{\alpha}\defeq\sup_{z_1,\,z_2\in\overline{Q}_T,\,z_1\neq z_2}\frac{|g_n(z_1)-g_n(z_2)|}{d_p(z_1,z_2)^{\alpha}}\rightarrow0.\]
For arbitrarily small $\delta>0$,
\begin{align*}
[g_n]_{\alpha}
&\leq\sup_{z_1\neq z_2,\,d_p(z_1,z_2)\leq\delta}\frac{|g_n(z_1)-g_n(z_2)|}{d_p(z_1,z_2)^{\alpha}}+\sup_{d_p(z_1,z_2)>\delta}\frac{|g_n(z_1)-g_n(z_2)|}{d_p(z_1,z_2)^{\alpha}}\\
&\leq\sup_{z_1\neq z_2,\,d_p(z_1,z_2)\leq\delta}\frac{|g_n(z_1)-g_n(z_2)|}{d_p(z_1,z_2)^{\beta}}d_p(z_1,z_2)^{\beta-\alpha}+\sup_{d_p(z_1,z_2)>\delta}\frac{|g_n(z_1)-g_n(z_2)|}{d_p(z_1,z_2)^{\alpha}}\\
&\leq\delta^{\beta-\alpha}[g_n]_{\beta}+2\delta^{-\alpha}\|g_n\|_{C(\overline{Q}_T)}\\
&\leq C\delta^{\beta-\alpha}+2\delta^{-\alpha}\|g_n\|_{C(\overline{Q}_T)}.
\end{align*}
Therefore, 
\[\limsup_{n\rightarrow\infty}\,[g_n]_{\alpha}\leq C\delta^{\beta-\alpha}+2\delta^{-\alpha}\limsup_{n\rightarrow\infty}\|g_n\|_{C(\overline{Q}_T)}\leq C\delta^{\beta-\alpha}.\]
Since $\delta>0$ is arbitrary, we conclude that $[g_n]_{\alpha}\rightarrow0$ as $n\rightarrow\infty$, which completes the proof.
\end{proof}

Also, we will use the following parabolic version of the Sobolev embedding theorem, which comes from the use of standard Sobolev embedding and an interpolation inequality.
\begin{proposition}\label{mainemb}
There exists a constant $C>0$, depending only on $d$ and $p$, such that
\begin{equation}\label{paraemb}
L^p(0,T;W^{1,p}_0(\Omega))\cap L^{\infty}(0,T;L^2(\Omega))\hookrightarrow L^{\frac{p(d+2)}{d}}(Q_T),
\end{equation}
with
\[\int_{Q_T}|v|^{\frac{p(d+2)}{d}}\dx\dt\leq C\left(\sup_{0\in(0,T)}\int_{\Omega}|v|^2\dx\right)^{\frac{p}{d}}\int_{Q_T}|\nabla v|^p\dx\dt\]
for every $v\in L^p(0,T;W^{1,p}_0(\Omega))\cap L^{\infty}(0,T;L^2(\Omega))$.
\end{proposition}

Additionally, we introduce another important embedding theorem, which is classical in the study of time-dependent problems. See, for example, \cite{S1987} or Lemma 6.3 in \cite{F2004}.
\begin{lemma}\label{aubinlion}
{\rm{(Aubin--Lions)}} Let $V_1$, $V_2$ and $V_3$ be reflexive, separable Banach spaces such that
\[V_1\hookrightarrow\hookrightarrow V_2\qquad{\rm{and}}\qquad V_2\hookrightarrow V_3,\]
and let $1<p<\infty$, $1\leq q\leq+\infty$. Then $\{v\in L^p(0,T;V_1) : \partial_tv\in L^q(0,T;V_3)\}$ is compactly embedded into $L^p(0,T;V_2)$.
\end{lemma}

Next, in order to ensure the H\"older continuity of the approximate concentrations in the later analysis, we will use the the following result, which is mainly due to De Giorgi, Nash and Moser;  see \cite{lady} for more details.
\begin{theorem}\label{degiorgi}
Suppose that there exist positive constants $C_1, C_2>0$ such that
\[\|\boldsymbol{K}\|_{L^{\infty}(Q_T)^{d\times d}}\leq C_1,\qquad\boldsymbol{K}\boldsymbol{b}\cdot\boldsymbol{b}\geq C_2|\boldsymbol{b}|^2\,\,\,{\rm{for}}\,\,{\rm{all}}\,\,\boldsymbol{b}\in\R^d.\]
Assume further that $\boldsymbol{g}\in L^r(0,T;L^q(\Omega)^d)$ where $q$ and $r$ are arbitrary constants satisfying
\[\frac{2}{r}+\frac{d}{q}=1-\varepsilon,\]
with
\[q\in\left[\frac{d}{1-\varepsilon},\infty\right],\,\,\,r\in\left[\frac{2}{1-\varepsilon},\infty\right]\,\,\,{\rm{and}}\,\,\,0<\varepsilon<1.\]
Let $c\in L^{\infty}(0,T;L^2(\Omega))\cap L^2(0,T;W^{1,2}_0(\Omega))$ be a weak solution of $\partial_tc-{\rm{div}}\,(\boldsymbol{K}(x,t)\nabla c(x,t))={\rm{div}}\,\boldsymbol{g}$ in the sense that
\[\langle\partial_tc,\varphi\rangle+\int_{\Omega}\boldsymbol{K}\nabla c\cdot\nabla\varphi=\int_{\Omega}\boldsymbol{g}\cdot\nabla\varphi\dx\qquad\forall\varphi\in W^{1,2}_0(\Omega).\]
If $c_{|_{\Gamma_T}}\in C^{\beta,\beta/2}(\Gamma_T)$ for some $\beta\in(0,1)$, then there exists an $\alpha\in(0,\beta]$ depending on $d$, $C_1$, $C_2$, $\beta$, $\boldsymbol{g}$, $q$ and $r$ such that
\[c\in C^{\alpha,\alpha/2}(\overline{Q}_T),\]
with
\[\|c\|_{C^{\alpha,\alpha/2}}\leq C(d,C_1,C_2,\beta,\|\boldsymbol{g}\|_{r,q}).\]\\

\end{theorem}

Now we move to function space theory for variable-exponent spaces. In recent years, the mathematical study of variable-exponent spaces $L^{p(\cdot)}$ has been an active research area. One of the major breakthroughs was the identification of the minimum regularity of the exponent $p(\cdot)$, which guarantees the possibility to extend various results from standard Lebesgue space theory to variable-exponent spcaes: {\textit{log-H\"older continuity}}. Most importantly, this regularity on $p(\cdot)$  ensures the boundedness of the Hardy--Littlewood maximal operator on $L^{p(\cdot)}$, which enables one to use tools from harmonic analysis. More precisely, we introduce a subset $\mathcal{P}^{\rm{log}}(Q_T)\subset\mathcal{P}(Q_T)$ as the class of log-H\"older continuous functions, satisfying
\begin{equation}\label{logh}
|p(z_1)-p(z_2)|\leq\frac{C_{\rm{log}}(p)}{-\log|z_1-z_2|}\,\,\,\,\,\forall z_1,z_2\in Q_T:0<|z_1-z_2|\leq\frac{1}{2}.
\end{equation}

The following lemma shows that a parabolic H\"older continuous function on $Q_T$ automatically belongs to this class.
\begin{lemma}\label{Hinclusion}
Let $z_1$, $z_2\in Q_T$ with $|z_1-z_2|<\frac{1}{8}$. Then, for any $\alpha\in(0,1)$, there exists a $C>0$ such that
\[d_p(z_1,z_2)^{\alpha}\leq\frac{C}{-{\rm{log}}\,|z_1-z_2|}.\]
\end{lemma}
\begin{proof}
From the definition of the parabolic distance, we have $d_p(z_1,z_2)<\frac{1}{2}$. Furthermore, by an elementary calculation, we can easily verify the following inequalities:
\[d_p(z_1,z_2)\leq2|z_1-z_2|^{\frac{1}{2}}\,\,\,{\rm{for}}\,\,\,|z_1-z_2|<1\qquad{\rm{and}}\qquad x^{\alpha}\leq\frac{C}{-{\rm{log}}\,x}\,\,\,{\rm{for}}\,\,\,0<x<\frac{1}{2}.\]
Therefore, we have
\[d_p(z_1,z_2)^{\alpha}\leq2^{\alpha}|z_1-z_2|^{\frac{\alpha}{2}}\leq\frac{C2^{\alpha}}{-{\rm{log}}\,|z_1-z_2|}.\]
\end{proof}

What we need in this paper is a density result for the function space $W_{p(\cdot)}(Q_T)$ defined in Section 2. First, from \cite{MNRR1996}, we infer the existence of a sequence $\{\lambda_j\}_{j=1}^{\infty}\subset\R$ and a sequence of functions $\{\ww_j\}^{\infty}_{j=1}\subset W^{\ell,2}_{0,\,{\rm{div}}}(\Omega)^d$, $\ell\in\mathbb{N}$, satisfying
\begin{itemize}
\item each $\ww_j$ is an eigenfunction associated with the eigenvalue $\lambda_j$ in the following sense:
\[\langle\ww_j,\boldsymbol{\phi}\rangle_{W^{\ell,2}_0}=\lambda_j\int_{\Omega}\ww_j\cdot \boldsymbol{\phi}\dx\qquad\forall\boldsymbol{\phi}\in W^{\ell,2}_{0,\rm{div}}(\Omega)^d,\]
\item $\int_{\Omega}\ww_i\cdot\ww_j\dx=\delta_{ij}$ for all $i$, $j\in\mathbb{N}$,
\item $1\leq\lambda_1\leq\lambda_2\leq\cdots$ and $\lambda_j\rightarrow\infty$,
\item $\langle\frac{\ww_i}{\sqrt{\lambda_i}}, \frac{\ww_j}{\sqrt{\lambda_i}}\rangle_{W^{\ell,2}_0}=\delta_{ij}$ for all $i$, $j\in\mathbb{N}$,
\item $\{\ww_j\}^{\infty}_{j=1}$ is a basis of $W^{\ell,2}_{0,\,{\rm{div}}}(\Omega)^d$.
\end{itemize}
We choose $\ell>1+\frac{d}{2}$ so that $W^{\ell,2}_{0,\,{\rm{div}}}(\Omega)^d\hookrightarrow W^{1,\infty}_{0,\,{\rm{div}}}(\Omega)^d$. Then we have the following lemma, which is quoted from \cite{AS2009}. 
\begin{lemma}\label{maindensity}
Let $p\in\mathcal{P}^{\rm{log}}(Q_T)$. For each $\uu\in W_{p(\cdot)}(Q_T)$, there exists a sequence $\{d_j(t)\}^{\infty}_{j=1}\subset C^1([0,T])$, satisfying
\[\norm{\uu-\sum^m_{j=1}d_j\,\ww_j}_{W_{p(\cdot)}}\rightarrow0\qquad{\rm{as}}\,\,\,m\rightarrow\infty.\]
\end{lemma}
In fact, unlike the space used in \cite{AS2009}, here we are dealing with the function space of solenoidal functions. However, that does not affect the proof of the lemma as the basis $\{\ww_j\}^{\infty}_{j=1}$ consists of solenoidal functions in this paper.

\end{section}

\begin{section}{Proof of the main theorem}
\begin{subsection}{Galerkin approximations}
Let $\{\ww_j\}^{\infty}_{j=1}$ be a basis of $W^{\ell,2}_{0,{\rm{div}}}(\Omega)^d,$ constructed in the previous section, and let $\{z_k\}^{\infty}_{k=1}$ be a basis of $W^{1,2}_0(\Omega)$ which is orthonormal in $L^2(\Omega)$. We are looking for a sequence of approximate solutions $\{\uu^{N,M},\,c^{N,M}\}^{\infty}_{N,M=1}$ of the form
\[\uu^{N,M}=\sum_{i=1}^{N}a_{i}^{N,M}\ww_i,\qquad c^{N,M}=\sum_{i=1}^{M}b^{N,M}_iz_i,\]
where $\boldsymbol{a}^{N,M}=\{a^{N,M}_j\}^N_{j=1}:(0,T)\rightarrow\R^N$ and $\boldsymbol{b}^{N,M}=\{b^{N,M}_k\}_{k=1}^M:(0,T)\rightarrow\R^M$, which solve the following system of ordinary differential equations:

\begin{equation}\label{Gmain1}
\int_{\Omega} \left( \partial_t\uu^{N,M}\cdot\ww_j-(\uu^{N,M}\otimes\uu^{N,M}):\DD\ww_j+\SSS^{N,M}:\DD\ww_j \right) \dx=\int_{\Omega}\boldsymbol{f}\cdot\ww_j\dx\,\,\,
\forall j=1,\ldots,N
\end{equation}

\begin{equation}\label{Gmain1_initial}
\uu^{N,M}(\cdot,0)=\uu^{N,M}_0=P^N\uu_0,
\end{equation}

\begin{equation}\label{Gmain2}
\int_{\Omega}\partial_t c^{N,M}\cdot z_k\dx-\int_{\Omega}\uu^{N,M}c^{N,M}\cdot\nabla z_k\dx+\int_{\Omega}\q^{N,M}_c\cdot\nabla z_k\dx=0\,\,\,\forall k=1,\ldots M,
\end{equation}

\begin{equation}\label{Gmain2_initial}
c^{N,M}(\cdot,0)=c^{N,M}_0=P^Mc_0,
\end{equation}
where 
\[\SSS^{N,M}=\SSS(c^{N,M},\DD\uu^{N,M}),\qquad\q_c^{N,M}=\q_c(c^{N,M},\nabla c^{N,M},\DD\uu^{N,M}),\]
and $P^{N}$, $P^{M}$ denote the $L^2$-orthogonal projections onto the finite-dimensional linear spaces $A_N=$ span$\{\ww_1,\ldots,\ww_N\}$, $B_M=$ span$\{z_1,\ldots,z_M\}$ respectively.

Carath\'eodory's theorem (see \cite{Z1990}, chapter 30) then guarantees the existence of solutions to \eqref{Gmain1}--\eqref{Gmain2_initial} at least for a short time interval. The uniform estimates that we will derive below enable us to extend the solution onto the whole time interval $(0,T)$.
\end{subsection}

\begin{subsection}{Uniform estimates independent of $M$}
We first derive some uniform estimates that are independent of $M$. Multiplying the $j$th equation in \eqref{Gmain1} by $a^{N,M}_j$, followed by taking the sum over $j=1,\ldots,N$ and integrating the result over $(0,t)$, we obtain
\[\|\uu^{N,M}(t)\|^2_2+2\int_0^t\int_{\Omega}\SSS^{N,M}:\DD\uu^{N,M}\dx\,\mathrm{d}\tau=\|\uu_0^{N,M}\|^2_2+2\int^t_0\int_{\Omega}\boldsymbol{f}\cdot\uu^{N,M}\dx\,\mathrm{d}\tau.\]
Applying \eqref{S3} and Poincar\'e's inequality and using standard duality estimates together with Young's inequality, we conclude that
\begin{equation}\label{mainest1}
\sup_{t\in(0,T)}\|\uu^{N,M}(t)\|^2_2+\int^T_0\int_{\Omega}\left(|\nabla\uu^{N,M}|^{p^-}+|\SSS^{N,M}|^{(p^+)'}\right)\dx\dt\leq C,
\end{equation}
and hence, by Proposition \ref{mainemb},
\begin{equation}\label{mainest1-3}
\|\uu^{N,M}\|_{L^{\frac{p^-(d+2)}{d}}(Q_T)}\leq C.
\end{equation}

Similarly, multiplying the $k$th equation in \eqref{Gmain2} by $b^{N,M}_k$, taking the sum over $k=1,\ldots,M$ and integrating the result over $(0,t)$, we arrive at
\[\|c^{N,M}(t)\|^2_2+2\int^t_0\int_{\Omega}\q_c^{N,M}\cdot\nabla c^{N,M}\dx\,\mathrm{d}\tau=\|c_0^{N,M}\|^2_2.\]
By using \eqref{q1} and \eqref{q2}, we have
\begin{equation}\label{mainest2}
\sup_{t\in(0,T)}\|c^{N,M}(t)\|^2_2+\int^T_0\int_{\Omega}\left(|\nabla c^{N,M}|^2+|\q_c^{N,M}|^2\right)\dx\dt\leq C.
\end{equation}

Next, we also need to estimate the time derivative of $\aaa^{N,M}$. We therefore multiply the $j$th equation in \eqref{Gmain1} by $\frac{{\mathrm{d}}}{{\mathrm{d}}t}a^{N,M}_j$, sum over $j=1,\ldots,N$, and integrate over time. It is obvious that the first term becomes $\int^T_0|\frac{{\mathrm{d}}}{{\mathrm{d}}t}\aaa^{N,M}|^2\dt$.
For the convective term, by \eqref{mainest1} we have
\begin{align*}
&\hspace{-15mm}\int_0^T\int_{\Omega}(\uu^{N,M}\otimes\uu^{N,M}):\left(\sum^{N}_{j=1}\frac{{\mathrm{d}}}{{\mathrm{d}}t}a^{N,M}_j\nabla\ww_j\right)\dx\dt\\
&\leq\int^T_0\int_{\Omega}|\uu^{N,M}\otimes\uu^{N,M}|\left(\sum^{N}_{j=1}\bigg|\frac{{\mathrm{d}}}{{\mathrm{d}}t}a^{N,M}_j\bigg|^2\right)^{\frac{1}{2}}\left(\sum^{N}_{j=1}|\nabla \ww_j|^2\right)^{\frac{1}{2}}\dx\dt\\
&\leq\int^T_0\left(\sum^{N}_{j=1}\bigg|\frac{{\mathrm{d}}}{{\mathrm{d}}t}a^{N,M}_j\bigg|^2\right)^{\frac{1}{2}}\int_{\Omega}|\uu^{N,M}\otimes\uu^{N,M}|\left(\sum^{N}_{j=1}|\nabla \ww_j|^2\right)^{\frac{1}{2}}\dx\dt\\
&\leq\left(\int^T_0\bigg|\frac{{\mathrm{d}}}{{\mathrm{d}}t}\aaa^{N,M}\bigg|^2\dt\right)^{\frac{1}{2}}\left(\int^T_0\left(\int_{\Omega}|\uu^{N,M}\otimes\uu^{N,M}|\left(\sum^{N}_{j=1}|\nabla \ww_j|^2\right)^{\frac{1}{2}}\dx\right)^2\dt\right)^{\frac{1}{2}}\\
&\leq C(N)\left(\int^T_0\bigg|\frac{{\mathrm{d}}}{{\mathrm{d}}t}\aaa^{N,M}\bigg|^2\dt\right)^{\frac{1}{2}}\left(\int^T_0\left(\int_{\Omega}|\uu^{N,M}\otimes\uu^{N,M}|\dx\right)^2\dt\right)^{\frac{1}{2}}\\
&\leq C(N)\left(\int^T_0\bigg|\frac{{\mathrm{d}}}{{\mathrm{d}}t}\aaa^{N,M}\bigg|^2\dt\right)^{\frac{1}{2}}\sup_{t\in(0,T)}\|\uu^{N,M}(t)\|^2_2\\
&\leq C(N)\left(\int^T_0\bigg|\frac{{\mathrm{d}}}{{\mathrm{d}}t}\aaa^{N,M}\bigg|^2\dt\right)^{\frac{1}{2}}.
\end{align*}
For the term involving the stress tensor, note that

\begin{align*}
&\hspace{-15mm}\int^T_0\int_{\Omega}\SSS^{N,M}:\left(\sum^{N}_{j=1}\frac{{\mathrm{d}}}{{\mathrm{d}}t}a^{N,M}_j\nabla \ww_j\right)\dx\dt\\
&\leq\int^T_0\int_{\Omega}|\SSS^{N,M}|\left(\sum^{N}_{j=1}\bigg|\frac{{\mathrm{d}}}{{\mathrm{d}}t}a^{N,M}_j\bigg|^2\right)^{\frac{1}{2}}\left(\sum^{N}_{j=1}|\nabla \ww_j|^2\right)^{\frac{1}{2}}\dx\dt\\
&\leq\int^T_0\left(\sum^{N}_{j=1}\bigg|\frac{{\mathrm{d}}}{{\mathrm{d}}t}a_j^{N,M}\bigg|^2\right)^{\frac{1}{2}}\int_{\Omega}|\SSS^{N,M}|\left(\sum^{N}_{j=1}|\nabla \ww_j|^2\right)^{\frac{1}{2}}\dx\dt\\
&\leq\left(\int^T_0\bigg|\frac{{\mathrm{d}}}{{\mathrm{d}}t}\aaa^{N,M}\bigg|^2\dt\right)^{\frac{1}{2}}\left(\int^T_0\left(\int_{\Omega}|\SSS^{N,M}|\left(\sum^{N}_{i=1}|\nabla \ww_i|^2\right)^{\frac{1}{2}}\dx\right)^2\dt\right)^{\frac{1}{2}}\\
&\leq C(N)\left(\int^T_0\bigg|\frac{{\mathrm{d}}}{{\mathrm{d}}t}\aaa^{N,M}\bigg|^2\dt\right)^{\frac{1}{2}}\left(\int^T_0\left(\int_{\Omega}|\SSS^{N,M}|\dx\right)^2\dt\right)^{\frac{1}{2}}.
\end{align*}
By equivalence of norms in finite-dimensional spaces,
\[\int_{\Omega}|\SSS^{N,M}|\dx\leq\int_{\Omega}\left(C_1+C_2|\DD\uu^{N,M}|^{p^+-1}\right)\dx\leq C+C(N)\|\uu^{N,M}\|^{p^+-1}_2.\]
Hence, by \eqref{mainest1}, we have
\begin{align*}
&\hspace{-15mm}\int^T_0\int_{\Omega}\SSS^{N,M}:\left(\sum^{N}_{j=1}\frac{d}{dt}a^{N,M}_j\nabla \ww_j\right)\dx\dt\\
&\leq C(N)\left(\int^T_0\bigg|\frac{d}{dt}\aaa^{N,M}\bigg|^2\dt\right)^{\frac{1}{2}}\left(\int^T_0\left(\int_{\Omega}|\SSS^{N,M}|\dx\right)^2\dt\right)^{\frac{1}{2}}\\
&\leq C(N)\left(\int^T_0\bigg|\frac{d}{dt}\aaa^{N,M}\bigg|^2\dt\right)^{\frac{1}{2}}\left(C+C(N)\int^T_0\|\uu^{N,M}\|^{2(p^+-1)}_2\dt\right)^{\frac{1}{2}}\\
&\leq C(N)\left(\int^T_0\bigg|\frac{d}{dt}\aaa^{N,M}\bigg|^2\dt\right)^{\frac{1}{2}}\left(C+C(N)\sup_{t\in(0,T)}\|\uu^{N,M}\|^{2(p^+-1)}_2\right)^{\frac{1}{2}}\\
&\leq C(N)\left(\int^T_0\bigg|\frac{d}{dt}\aaa^{N,M}\bigg|^2\dt\right)^{\frac{1}{2}}.
\end{align*}
Finally for the external force term,
\begin{align*}
\int^T_0\int_{\Omega}\boldsymbol{f}\cdot\left(\sum^{N}_{j=1}\frac{d}{dt}a^{N,M}_j\ww_j\right)\dx\dt
&\leq\int^T_0\int_{\Omega}|\boldsymbol{f}|\left(\sum^N_{j=1}\bigg|\frac{{\mathrm{d}}}{{\mathrm{d}}t}a^{N,M}_j\bigg|^2\right)^{\frac{1}{2}}\left(\sum^N_{j=1}|\ww_j|^2\right)^{\frac{1}{2}}\dx\dt\\
&\leq C(N)\int^T_0\left(\sum^N_{j=1}\bigg|\frac{{\mathrm{d}}}{{\mathrm{d}}t}a^{N,M}_j\bigg|^2\right)^{\frac{1}{2}}\|\boldsymbol{f}\|_2\dt\\
&\leq C(N)\|\boldsymbol{f}\|_{L^2(Q_T)}\left(\int^T_0\bigg|\frac{{\mathrm{d}}}{{\mathrm{d}}t}\aaa^{N,M}\bigg|^2\dt\right)^{\frac{1}{2}}.
\end{align*}
Collecting all the terms above gives us the following uniform estimate, independent of $M$:
\begin{equation}\label{mainest3}
\int^T_0\bigg|\frac{{\mathrm{d}}}{{\mathrm{d}}t}\aaa^{N,M}\bigg|^2\dt\leq C(N).
\end{equation}

As a last step, we shall derive an estimate for the time derivative of the approximate concentration. Let $P^M_1$ be the orthogonal projection onto $B_M$ with respect to the $W^{1,2}_0(\Omega)$ inner product. Then, for any $\varphi\in W^{1,2}_0(\Omega)$, we have from \eqref{Gmain2} that
\begin{align*}
\bigg|\int_{\Omega}\partial_tc^{N,M}\varphi\dx\bigg|
&=\bigg|\int_{\Omega}\partial_tc^{N,M}P^M_1\varphi\dx\bigg|\\
&\leq\int_{\Omega}|\uu^{N,M}c^{N,M}||\nabla P^M_1\varphi|\dx+\int_{\Omega}|\q_c^{N,M}||\nabla P^M_1\varphi|\dx\\
&\leq\|\uu^{N,M}c^{N,M}\|_2\|P^M_1\varphi\|_{1,2}+\|\q_c^{N,M}\|_2\|P^M_1\varphi\|_{1,2}\\
&\leq C\left(\|\uu^{N,M}c^{N,M}\|_2+\|\q_c^{N,M}\|_2\right)\|\varphi\|_{1,2}.
\end{align*}
Integrating over the time interval $(0,T)$ together with \eqref{mainest1}, \eqref{mainest2}, Sobolev embedding and norm equivalence in finite-dimensional spaces leads us to
\begin{align*}
\int^T_0\norm{\partial_tc^{N,M}}^2_{W^{-1,2}(\Omega)}\dt
&\leq C\int^T_0(\|\uu^{N,M}c^{N,M}\|^2_2\dt+\|\q_c^{N,M}\|^2_2)\dt\\
&\leq C\int^T_0\|\uu^{N,M}\|^2_3\|c^{N,M}\|^2_6\dt+C\\
&\leq C(N)\int^T_0\|\uu^{N,M}\|^2_2\|\nabla c^{N,M}\|^2_2\dt+C\\
&\leq C(N).
\end{align*}
Therefore, we finally have
\begin{equation}\label{mainest4}
\|\partial_tc^{N,M}\|_{L^2(0,T;W^{-1,2}(\Omega))}\leq C(N).
\end{equation}
\end{subsection}

\begin{subsection}{The limit $M\rightarrow\infty$}
Having shown the uniform estimates \eqref{mainest1}--\eqref{mainest4}, we can establish some necessary convergence results for a selected (not relabelled) subsequence, as $M\rightarrow\infty$. First, by \eqref{mainest3},
\begin{equation}\label{1dconv}
\aaa^{N,M}\rightharpoonup\aaa^{N}\,\,\,{\rm{weakly}}\,\,{\rm{in}}\,\,W^{1,2}(0,T)\,\,{\rm{and}}\,\,{\rm{thus}}\,\,{\rm{strongly}}\,\,{\rm{in}}\,\,C([0,T]),
\end{equation}
from which we obtain the following uniform convergences results:
\begin{equation}\label{Mconv1}
\uu^{N,M}\rightrightarrows\uu^{N}\,\,\,{\rm{and}}\,\,\,\DD\uu^{N,M}\rightrightarrows\DD\uu^{N}.
\end{equation}

Next, by \eqref{mainest2} and \eqref{mainest4},
\begin{equation}\label{Mconv2}
c^{N,M}\rightharpoonup c^{N}\,\,\,{\rm{weakly}}\,\,\,{\rm{in}}\,\,\,\left\{z\in L^2(0,T;W^{1,2}_0(\Omega)),\,\,\partial_tz\in L^2(0,T;W^{-1,2}(\Omega))\right\}.
\end{equation}
Therefore by applying the Aubin--Lions lemma, we deduce that
\begin{equation}\label{ML2}
c^{N,M}\rightarrow c^{N}\,\,\,{\rm{strongly}}\,\,{\rm{in}}\,\,L^2(Q_T),
\end{equation}
which implies, for a (relabelled) subsequence,
\begin{equation}\label{lkaeconv}
c^{N,M}\rightarrow c^{N}\,\,\,{\rm{almost}}\,\,{\rm{everywhere}}\,\,{\rm{in}}\,\,Q_T.
\end{equation}

Next, from \eqref{Mconv1} and \eqref{lkaeconv},
\[\SSS^{N,M}\rightarrow \SSS^{N}\defeq\SSS(c^N,\DD\uu^N)\,\,\,{\rm{almost}}\,\,{\rm{everywhere}}\,\,{\rm{in}}\,\,Q_T.\]
Also, by \eqref{Mconv1} again, for sufficiently large $M\in\mathbb{N}$,
\[|\DD\uu^{N,M}|<1+|\DD\uu^N|\,\,\,{\rm{almost}}\,\,{\rm{everywhere}}\,\,{\rm{in}}\,\,\,Q_T.\]
Therefore, by \eqref{S1}, we have, for sufficiently large $M\in\mathbb{N}$, that
\begin{align*}
|\SSS^{N,M}|
&\leq C|\DD\uu^{N,M}|^{p(c^{N,M})-1}+C\\
&\leq C(1+|\DD\uu^N|)^{p(c^{N,M})-1}+C\\
&\leq C(1+|\DD\uu^N|)^{p^+-1}+C\in L^{(p^+)'}(Q_T).
\end{align*}
Thus, by the dominated convergence theorem, we have
\begin{equation}\label{SNMconv}
\SSS^{N,M}\rightarrow\SSS^N\,\,\,{\rm{strongly}}\,\,\,{\rm{in}}\,\,\,L^{(p^+)'}(Q_T)^{d\times d}.
\end{equation}
Finally, concerning $\q_c^{N,M}=\boldsymbol{K}(c^{N,M},|\DD\uu^{N,M}|)\nabla c^{N,M}$, by \eqref{lkaeconv}, \eqref{Mconv1} and the dominated convergence theorem,
\[\boldsymbol{K}(c^{N,M},|\DD\uu^{N,M}|)\rightarrow\boldsymbol{K}(c^{N},\DD\uu^{N})\,\,\,{\rm{strongly}}\,\,{\rm{in}}\,\,L^2(Q_T).\]
Therefore, together with \eqref{Mconv2},
\begin{equation}\label{lkconv6}
\q_c^{N,M}\rightharpoonup\q^{N}_c\defeq\q_c(c^{N},\nabla c^{N},\DD\uu^{N})\,\,\,{\rm{weakly}}\,\,{\rm{in}}\,\,L^2(Q_T)^d.
\end{equation}

The convergence results established above allow us to take the limit in our Galerkin approximation. First, by using \eqref{1dconv}, \eqref{Mconv1} and \eqref{SNMconv} we have
\begin{equation}\label{Gsecond1}
\int_{\Omega} \left( \partial_t\uu^{N}\cdot\ww_j-(\uu^{N}\otimes\uu^{N}):\DD\ww_j+\SSS^{N}:\DD\ww_j \right) \dx=\int_{\Omega}\boldsymbol{f}\cdot\ww_j\dx\,\,\,
\forall j=1,\ldots,N.
\end{equation}
Regarding the initial condition for the velocity, thanks to \eqref{1dconv} and the fact that $\aaa^{N,M}(0)=\aaa^N_0$ for all $M$, if we denote $\lim_{M\rightarrow\infty}\uu^{N,M}_0=\uu^N_0=P^N\uu_0$, we have
\[\uu^N(\cdot,0)=\uu^N_0.\]

Next, from \eqref{Mconv1}, \eqref{Mconv2}, \eqref{lkconv6} and a density argument, we have for a.e. $t\in(0,T)$ that
\begin{equation}\label{Gsecond2}
\langle\partial_t c^{N},\varphi\rangle-\int_{\Omega}\left(\uu^{N}c^{N}\cdot\nabla \varphi-\q_c^{N}\cdot\nabla \varphi\right)\dx=0\qquad\forall \varphi\in W^{1,2}_0(\Omega).
\end{equation}
We shall verify the attainment of the initial condition for the concentration. First, by \eqref{mainest2} and \eqref{mainest4}, we have
\begin{equation}\label{timeconti}
c^N\in C([0,T];L^2(\Omega)).
\end{equation}
Next, we integrate \eqref{Gmain2} over $t\in(0,\tau)$, and pass $M$ to the limit to get
\[\int_{\Omega}c^N(\tau)z_k\dx-\int^{\tau}_0\int_{\Omega}\left(\uu^Nc^N\cdot\nabla z_k-\q_c^N\cdot\nabla z_k\right)\dx\dt=\int_{\Omega}c_0z_k\dx,\]
where we have used that $\|c^{N,M}_0-c_0\|_2=\|P^Mc_0-c_0\|_2\rightarrow0$ as $M\rightarrow\infty$.
Since $\int_{\Omega}\uu^Nc^N\cdot\nabla z_k\dx$ and $\int_{\Omega}\q_c^N\cdot\nabla z_k\dx$ are integrable in time, as $\tau\rightarrow0$, we have
\[\int_{\Omega}c^N(\tau)z_k\dx\rightarrow\int_{\Omega}c_0z_k\dx.\]
Hence, together with \eqref{timeconti}, we finally obtain
\[\lim_{t\rightarrow0+}\|c^N(t)-c_0\|_2=0.\]
By \eqref{timeconti} again, we deduce that
\[c^N(0)=c_0\,\,\,{\rm{in}}\,\,\,L^2(\Omega).\]
\end{subsection}

\begin{subsection}{Maximum and minimum principles}
Before proceeding further, we shall derive maximum and minimum principles for the concentration by using standard tools for parabolic problems.

We begin with the maximum principle. Let $\varphi\defeq\chi_{[0,\tau]}(t)\max\{0,c^N-\tilde{c}_0\}$ where $\tilde{c}_0$ is the upper bound on the initial datum $c_0$. Clearly $\varphi\in W^{1,2}_0(\Omega)$ for a.e. $t\in(0,T$), so we test by $\varphi$ in \eqref{Gsecond2}. Integrating over the time $t\in(0,T)$ yields
\[0=\int^T_0\langle\partial_tc^N,\varphi\rangle\dt-\int^T_0\int_{\Omega}\left(\uu^Nc^N\cdot\nabla\varphi-\q^N_c\cdot\nabla\varphi\right)\dx\dt\eqdef{\rm{I}}+{\rm{II}}+{\rm{III}},\]
with an obvious numbering. We can then compute I, II and III as follows:
\begin{align*}
{\rm{I}}&=\int^T_0\langle\partial_tc^N,\varphi\rangle\dt=\int_0^{\tau}\langle\partial_t\varphi,\varphi\rangle\dt=\frac{1}{2}\|\varphi(\tau)\|^2_2-\frac{1}{2}\|\varphi(0)\|^2_2,\\
{\rm{II}}&=-\int^T_0\int_{\Omega}\uu^Nc^N\cdot\nabla\varphi\dx\dt=\int^T_0\int_{\Omega}\uu^N\cdot\nabla c^N\varphi\dx\dt\\
&=\int^T_0\int_{\Omega}\uu^N\cdot\nabla\varphi\,\varphi\dx\dt=\frac{1}{2}\int^T_0\int_{\Omega}\uu^N\cdot\nabla\,|\varphi|^2\dx\dt=0,\\
{\rm{III}}&=\int^T_0\int_{\Omega}\q^N_c\cdot\nabla\varphi\dx\dt=\int_0^T\int_{\Omega}\boldsymbol{K}(c^N,|\DD\uu^N|)\nabla c^N\cdot\nabla\varphi\dx\dt\\
&=\int_0^T\int_{\Omega}\boldsymbol{K}(c^N,|\DD\uu^N|)\nabla\varphi\cdot\nabla\varphi\dx\dt\geq0.
\end{align*}
Altogether, we have $\|\varphi(\tau)\|_2\leq\|\varphi(0)\|_2=0$, which implies that \begin{equation}\label{maxp}
c^N(x,t)\leq\tilde{c}_0\,\,\,{\rm{a.e.}}\,\, {\rm{in}}\,\,\, Q_T.
\end{equation}

In the case of the minimum principle, we choose $\varphi\defeq\chi_{[0,\tau]}(t)\min\{0,c^N\}$, which obviously belongs to $W^{1,2}_0(\Omega)$ for a.e. $t\in(0,T)$. If we proceed as above, we obtain that
\begin{equation}\label{minp}
0\leq c^N(x,t)\,\,\,{\rm{a.e.}}\,\, {\rm{in}}\,\,\, Q_T.
\end{equation}

 Note that since convex sets are weakly closed and the set $\{x\in\R:x\geq 0\}$ is convex, weak or strong limits of $c^N$ are still non-negative.
\end{subsection}

\begin{subsection}{Uniform estimates independent of $N$}
Now we shall derive some uniform estimates independent of $N$. We first multiply the $j$th equation in \eqref{Gsecond1} by $a^{N}_j$, and then take the sum over $j=1,\ldots,N$. If we integrate the result over $(0,t)$, we can deduce that
\begin{equation}\label{secest1}
\sup_{t\in(0,T)}\|\uu^{N}(t)\|^2_2+\int^T_0\int_{\Omega}\left(|\nabla\uu^{N}|^{p(c^N)}+|\SSS^{N}|^{p'(c^N)}\right)\dx\dt\leq C.
\end{equation}
By Proposition \ref{mainemb}, we have
\begin{equation}\label{secest1-2}
\|\uu^{N}\|_{L^{\frac{p^-(d+2)}{d}}(Q_T)}\leq C,
\end{equation}
and thus
\begin{equation}\label{secest1-3}
\|\uu^{N}\otimes\uu^N\|_{L^{\frac{p^-(d+2)}{2d}}(Q_T)}\leq C.
\end{equation}
Furthermore, by weak and weak-* lower semicontinuity of norms, we have from \eqref{mainest2} that
\begin{equation}\label{secest2}
\|c^N\|_{L^{\infty}(0,T;L^2(\Omega))}+\|c^N\|_{L^2(0,T;W^{1,2}_0(\Omega))}+\|\q^N_c\|_{L^2(0,T;L^2(\Omega))}\leq C.
\end{equation}

Now we shall derive a uniform H\"older estimate for the approximate concentrations $c^N$. Since $c_0\in \{f\in C^{\alpha_0}(\overline{\Omega})\,\,\,{\rm{for}}\,\,\,{\rm{some}}\,\,\,\alpha_0\in(0,1):f=0\,\,\,{\rm{on}}\,\,\,\partial\Omega\}$ and $c^N(\cdot,0)=c_0$ a.e. in $\Omega$ for all $N\in\mathbb{N}$, we have that for some $\alpha_0\in(0,1)$ independent of $N$, $c^N(\cdot,0)\in C^{\alpha_0}(\overline{\Omega})$ after possibly being redefined on a set of measure zero. Therefore we have $c^N_{|_{\Gamma_T}}\in C^{\beta,\beta/2}(\Gamma_T)$ for some $\beta\in(0,1)$ independent of $N$. 

Now, we are ready to apply Theorem \ref{degiorgi}. We first consider the case of $d=2$. In this case, $p^->\frac{d+2}{2}=2$ implies the existence of a positive real number $m$ sufficietly close to $4$ so that $2p^->m>4$ and
\[\frac{2}{m}+\frac{d}{m}=\frac{2}{m}+\frac{2}{m}=1-\varepsilon_1,\,\,\,{\rm{and}}\,\,\,m>\frac{2}{1-\varepsilon_1}\,\,\,{\rm{for}}\,\,\,{\rm{some}}\,\,\,{\rm{small}}\,\,\,\varepsilon_1>0.\]
Also, as we see from \eqref{maxp} and \eqref{minp} that $c^N$ is uniformly bounded, and hence we have from \eqref{secest1-2} that
\[\|\uu^Nc^N\|_{L^{m}(Q_T)}\leq C\|\uu^{N}\|_{{L^{2p^-}}(Q_T)}\leq C.\]
We then apply Theorem \ref{degiorgi} to \eqref{Gsecond2} with $m=r=q$ and $\boldsymbol{g}=\uu^Nc^N$. Therefore we have the following estimate when $d=2$:
\begin{equation}\label{secest4}
\|c^N\|_{C^{\alpha_1,\alpha_1/2}(\overline{Q}_T)}\leq C,
\end{equation}
where $\alpha_1\in(0,\beta]$ and $C>0$ is independent of $N$.

Next, we shall deal with the case of $d\geq3$, which is more complicated. From \eqref{secest1} and Sobolev embedding, we have
\begin{equation}\label{DeGuu}
\|\uu^N\|_{L^{p^-}(0,T;L^{\frac{dp^-}{d-p^-}}(\Omega)^d)}\leq C.
\end{equation}
Note that the assumption $p^->\frac{d+2}{2}$ is equivalent to $\frac{d(d+2)}{d-2}<\frac{dp^-}{d-p^-}$, hence we can choose sufficiently small $\delta>0$ so that $\frac{d(d+2)}{d-2}+\delta<\frac{dp^-}{d-p^-}$. Then we have from \eqref{DeGuu} that
\begin{equation}\label{DeGuu2}
\|\uu^N\|_{L^{p^-}(0,T;L^{\frac{d(d+2)}{d-2}+\delta}(\Omega)^d))}\leq C.
\end{equation}
Furthermore, if we choose $r=p^->2$ and $q=\frac{d(d+2)}{d-2}+\delta>d$, $p^->\frac{d+2}{2}$ implies that
\[\frac{2}{r}+\frac{d}{q}<\frac{2}{p^-}+\frac{d-2}{d+2}<\frac{4}{d+2}+\frac{d-2}{d+2}=1.\]
Now, from \eqref{maxp}, \eqref{minp} and \eqref{DeGuu2}
\begin{equation}\label{DeGuu3}
\|\uu^Nc^N\|_{L^r(0,T;L^q(\Omega)^d)}\leq C.
\end{equation}
We can then apply Theorem \ref{degiorgi} to \eqref{Gsecond2} with aforementioned $r>0$ and $q>0$, and $\boldsymbol{g}=\uu^Nc^N$. As a result, we obtain the following estimate for $d\geq3$:
\begin{equation}\label{DeGuu4}
\|c^N\|_{C^{\alpha_1,\alpha_1/2}(\overline{Q}_T)}\leq C,
\end{equation}
where $\alpha_1\in(0,\beta]$ and $C>0$ is independent of $N$.

Finally, we shall derive a uniform estimate for $\partial_t\uu^N$. For simplicity, we shall denote $\boldsymbol{H}^N\defeq-\SSS^N+\uu^N\otimes\uu^N$ and let $q_0\defeq\min\left\{\frac{p^-(d+2)}{2d},(p^+)'\right\}>1$. Then, from \eqref{secest1} and \eqref{secest1-3}, we obtain
\[\|\boldsymbol{H}^N\|_{L^{q_0}(Q_T)}\leq C.\]
Now let $\boldsymbol{\psi}\in W^{\ell,2}_{0,{\rm{div}}}(\Omega)^d$ and let $P^N_{\ell}$ denote the orthogonal projection into $A_N$ with respect to the $W^{\ell,2}_0(\Omega)$-inner product. Then, from \eqref{Gsecond1}, we have
\begin{align*}
\bigg|\int_{\Omega}\partial_t\uu^N\cdot\boldsymbol{\psi}\dx\bigg|
&=\bigg|\int_{\Omega}\partial_t\uu^N\cdot P^N_{\ell}\boldsymbol{\psi}\dx\bigg|\leq\int_{\Omega}|\boldsymbol{H}^N||\DD P^N_{\ell}\boldsymbol{\psi}|\dx+\int_{\Omega}|\boldsymbol{f}||P^N_{\ell}\boldsymbol{\psi}|\dx\\
&\leq\|\boldsymbol{H}^N\|_{q_0}\|P^N_{\ell}\boldsymbol{\psi}\|_{1,q_0'}+\|\boldsymbol{f}\|_2\|P^N_{\ell}\boldsymbol{\psi}\|_2\leq C\|\boldsymbol{H}^N\|_{q_0}\|P^N_{\ell}\boldsymbol{\psi}\|_{\ell,2}+\|\boldsymbol{f}\|_2\|P^N_{\ell}\boldsymbol{\psi}\|_{\ell,2}\\
&\leq C\left(\|\boldsymbol{H}^N\|_{q_0}+\|\boldsymbol{f}\|_2\right)\|\boldsymbol{\psi}\|_{\ell,2}.
\end{align*}

Therefore we have that
\[\int^T_0\|\partial_t\uu^N\|^{q_0}_{W^{-\ell,2}_{\rm{div}}(\Omega)}\dt\leq C\int^T_0\|\boldsymbol{H}^N\|^{q_0}_{q_0}\dt+\int^T_0\|\boldsymbol{f}\|^{q_0}_2\dt\leq C\left(\|\boldsymbol{H}^N\|^{q_0}_{L^{q_0}(Q_T)}+\|\boldsymbol{f}\|^2_{L^2(Q_T)}\right)\leq C,\]
which means that
\begin{equation}\label{secest5}
\|\partial_t\uu^N\|_{L^{q_0}(0,T;W^{-\ell,2}_{\rm{div}}(\Omega))}\leq C.
\end{equation}
\end{subsection}

\begin{subsection}{The limit $N\rightarrow\infty$}
By using the uniform estimates derived in the previous subsection, we can now pass $N$ to $\infty$. First of all, from \eqref{secest1} we have
\begin{equation}\label{seconv1}
\uu^N\rightharpoonup \uu\,\,\,{\rm{weakly}}\,\,{\rm{in}}\,\,L^{p^-}(0,T;W^{1,p^-}_{0,{\rm{div}}}(\Omega)^d).
\end{equation}
Furthermore, from \eqref{secest1} and \eqref{secest5} together with the Aubin--Lions compactness lemma, we have $\uu^N\rightarrow\uu$ strongly in $L^2(Q_T)^d$. Therefore, using \eqref{secest1-2} with an interpolation inequality yields
\begin{equation}\label{seconv2}
\uu^N\rightarrow\uu\,\,\,{\rm{strongly}}\,\,{\rm{in}}\,\,L^s(Q_T)^d\qquad\forall s\in[1,p^-(d+2)/d).
\end{equation}
From \eqref{secest5}, we also have
\begin{equation}\label{timeconv1}
\partial_t\uu^N\rightharpoonup\partial_t\uu\,\,\,{\rm{weakly}}\,\,{\rm{in}}\,\,L^{q_0}(0,T;W^{-\ell,2}_{\rm{div}}(\Omega)).
\end{equation}
By \eqref{secest1} again, we obtain
\begin{equation}\label{seconv3}
\SSS^N\rightharpoonup\bar{\SSS}\,\,\,{\rm{weakly}}\,\,{\rm{in}}\,\,L^{p^+}(Q_T)^{d\times d},
\end{equation}
for some $\bar{\SSS}\in\R^{d\times d}$. 

Next, we move to the convergence of the sequence of approximate concentrations. First, by \eqref{secest2}, we have
\begin{equation}\label{seconv4}
c^N\rightharpoonup c\,\,\,{\rm{weakly}}\,\,{\rm{in}}\,\,L^2(0,T;W^{1,2}_0(\Omega)).
\end{equation}
Furthermore, from \eqref{secest4} with Proposition \ref{Holdercpt}, we obtain
\begin{equation}\label{seconv5}
c^N\rightarrow c\,\,\,{\rm{strongly}}\,\,{\rm{in}}\,\,C^{\alpha_1,\alpha_1/2}(\overline{Q}_T)
\end{equation}
for some $\alpha_1\in(0,\alpha)$. Using \eqref{secest2} again yields
\begin{equation}\label{seconv6}
\q^N_c\rightharpoonup\bar{\q}_c\,\,\,{\rm{weakly}}\,\,{\rm{in}}\,\,L^2(Q_T)^d
\end{equation}
for some $\bar{\q}_c\in\R^d$.

Now we are ready to pass $N$ to $\infty$ in the second level of our Galerkin approximations \eqref{Gsecond1} and \eqref{Gsecond2}. First, we fix $\vv=g\ww_j$ where $g\in C^1_0([0,T))$. Then, by \eqref{seconv2} and \eqref{seconv3}, we have 
\begin{equation}\label{midGconv}
\int_{Q_T}\bar{\SSS}:\DD\vv\dx\dt
=\int_{Q_T}\left(\boldsymbol{f}\cdot\vv+(\uu\otimes\uu):\DD\vv+\uu\cdot\partial_t\vv\right)\dx\dt+\int_{\Omega}\uu_0(x)\cdot\vv(x,0)\dx,
\end{equation}
where we have used
\[\int_{Q_T}\partial_t\uu^N\cdot\vv\dx\dt=-\int_{\Omega}\uu^N_0\cdot\vv(x,0)\dx\dt-\int_{Q_T}\uu^N\cdot\partial_t\vv\dx\dt\]
and $\uu^N_0\rightarrow\uu_0$ strongly in $L^2(\Omega)^d$. Since the class of test functions which are finite linear combinations of functions that can be factorized in space and time is dense in the corresponding Bochner space, we have
\begin{equation}\label{Gfinal1}
\int_{Q_T}\bar{\SSS}:\DD\vv\dx\dt
=\int_{Q_T}\left(\boldsymbol{f}\cdot\boldsymbol{\psi}+(\uu\otimes\uu):\DD\boldsymbol{\psi}+\uu\cdot\partial_t\boldsymbol{\psi}\right)\dx\dt+\int_{\Omega}\uu_0(x)\cdot\boldsymbol{\psi}(x,0)\dx,
\end{equation}
for all $\boldsymbol{\psi}\in C^{\infty}_{0,{\rm{div}}}(\Omega\times[0,T))^d$.

With the same argument, \eqref{Gsecond2} together with \eqref{seconv2}, \eqref{seconv5} and \eqref{seconv6} implies
\[\int_{Q_T}\bar{\q}_c\cdot\nabla\varphi\dx\dt=\int_{Q_T}\left(c\,\partial_t\varphi+c\uu\cdot\nabla\varphi\right)\dx\dt+\int_{\Omega}c_0(x)\varphi(x,0)\dx\]
for all $\varphi\in C^{\infty}_0(\Omega\times[0,T))$.

Finally, we shall show that the limit function $\uu$ is contained in the desired solution space $W_{p(c)}(Q_T)$. By \eqref{seconv5} and the continuity of $p$,
\[\forall\varepsilon>0,\,\,\,\exists N\in\mathbb{N}\,\,\,{\rm{such}}\,\,{\rm{that}}\,\,\,n\geq\mathbb{N}\,\,\,{\rm{implies}}\,\,\,|p(c^N)-p(c)|<\frac{\varepsilon}{\theta},\]
where $\theta>1$ is sufficiently large so that $p(c)-\frac{\theta+1}{\theta}\varepsilon>1$. Then we have from \eqref{secest1} that
\begin{align*}
C&\geq\int_{Q_T}|\nabla\uu^N|^{p(c^N)}\dx\dt\geq\int_{\{|\nabla\uu^N|\geq1\}}|\nabla\uu^N|^{p(c^N)}\dx\dt\\
&\geq\int_{\{|\nabla\uu^N|\geq1\}}|\nabla\uu^N|^{p(c^N)-p(c)+p(c)-\varepsilon}\dx\dt\geq\int_{\{|\nabla\uu^N|\geq1\}}|\nabla\uu^N|^{p(c)-\frac{\theta+1}{\theta}\varepsilon}\dx\dt.
\end{align*}
Thus, we obtain that
\[\int_{Q_T}|\nabla\uu^N|^{p(c)-\frac{\theta+1}{\theta}\varepsilon}\dx\dt=\int_{\{|\nabla\uu^N|\geq1\}}|\nabla\uu^N|^{p(c)-\frac{\theta+1}{\theta}\varepsilon}\dx\dt+\int_{\{|\nabla\uu^N|<1\}}|\nabla\uu^N|^{p(c)-\frac{\theta+1}{\theta}\varepsilon}\dx\dt\leq C.\]
We can then further extract a (not relabelled) subsequence such that
\[\nabla\uu^N\rightharpoonup\nabla\uu\,\,\,{\rm{weakly}}\,\,\,{\rm{in}}\,\,\,L^{p(c)-\frac{\theta+1}{\theta}\varepsilon}(Q_T)^{d\times d}.\]
Then, weak lower-semicontinuity leads us to
\[\int_{Q_T}|\nabla\uu|^{p(c)-\frac{\theta+1}{\theta}\varepsilon}\dx\dt\leq C,\]
and thus Fatou's Lemma with $\varepsilon\rightarrow0$ yields
\begin{equation}\label{desnabla}
\int_{Q_T}|\nabla\uu|^{p(c)}\dx\dt\leq C.
\end{equation}
Similarly, we can also deduce that
\begin{equation}\label{desS}
\int_{Q_T}|\bar{\SSS}|^{p'(c)}\dx\dt\leq C.
\end{equation}

To apply the monotone operator theory in the next subsection, the limit function $\uu$ should be an admissible test function in \eqref{Gfinal1}. This is where the compactness of the convective term becomes relevant, and thus the value of $p^-$ should be restricted. Let us fix $\vv=g\ww$ where $g\in C^1([0,T])$ and $\ww\in\{\ww_j\}^{\infty}_{j=1}$. Then, from \eqref{Gsecond1}, we have
\[\int_{Q_T}\partial_t\uu^N\cdot\vv\dx\dt=\int_{Q_T}\left((\uu^N\otimes\uu^N):\nabla\vv-\SSS^N:\nabla\vv+\boldsymbol{f}\cdot\vv\right)\dx\dt.\]
From \eqref{timeconv1}, \eqref{seconv2} and \eqref{seconv3}, we have
\begin{equation}\label{8282}
\int^T_0\langle\partial_t\uu,\vv\rangle=\int_{Q_T}\left((\uu\otimes\uu):\nabla\vv-\bar{\SSS}:\nabla\vv+\boldsymbol{f}\cdot\vv\right)\dx\dt.
\end{equation}
Now let $\boldsymbol{\phi}\in W_{p(c)}(Q_T)$. Then, by Lemma \ref{maindensity} and \eqref{8282}, there exists a subset $\{\boldsymbol{\phi}_k\}_{k=1}^{\infty}\subset$ span$\{g\ww:g\in C^1([0,T])\,\,\,{\rm{and}}\,\,\,\ww\in\{\ww_j\}^{\infty}_{j=1}\}$ such that $\|\boldsymbol{\phi}_k-\boldsymbol{\phi}\|_{W_{p(c)}(Q_T)}\rightarrow0$ and
\begin{equation}\label{8282-3}
\int^T_0\langle\partial_t\uu,\boldsymbol{\phi}_k\rangle=\int_{Q_T}\left((\uu\otimes\uu):\nabla\boldsymbol{\phi}_k-\bar{\SSS}:\nabla\boldsymbol{\phi}_k+\boldsymbol{f}\cdot\boldsymbol{\phi}_k\right)\dx\dt.
\end{equation}
Since $p^->\frac{d+2}{2}\geq\frac{3d+2}{d+2}$ for $d\geq2$, we have from \eqref{secest1-2} that
\[\|\uu\otimes\uu\|_{L^{p'(c)}(Q_T)}\leq\|\uu\|^2_{L^{2p'(c)}(Q_T)}\leq C\|\uu\|^2_{L^{2(p^-)'}(Q_T)}\leq C\|\uu\|^2_{L^{\frac{p^-(d+2)}{d}}(Q_T)}\leq C,\]
and hence we can pass the first term on the right-hand of \eqref{8282-3} to the limit. Also, by \eqref{desS}, we may pass $k$ to the limit in the second term and the convergence of the third term is trivial. This means that the left-hand side also has a limit as $k\rightarrow\infty$. Therefore, we obtain that
\begin{equation}\label{Gthird1}
\int^T_0\langle\partial_t\uu,\boldsymbol{\phi}\rangle=\int_{Q_T}\left((\uu\otimes\uu):\nabla\boldsymbol{\phi}-\bar{\SSS}:\nabla\boldsymbol{\phi}+\boldsymbol{f}\cdot\boldsymbol{\phi}\right)\dx\dt\qquad\forall \boldsymbol{\phi}\in W_{p(c)}(Q_T),
\end{equation}
and
\begin{equation}\label{timeest2}
\partial_t\uu\in (W_{p(c)}(Q_T))^*.
\end{equation}

Next, we claim that $\uu\in C_{\rm{w}}([0,T];W^{-1,(p^+)'}(\Omega))$. As we have that $\sup_{t\in(0,T)}\|\uu\|_{W^{-1,(p^+)'}}\leq C\sup_{t\in(0,T)}\|\uu(t)\|_2\leq C$, it is enough to show that $\int_{\Omega}(\uu(t_1)-\uu(t_0))\cdot\boldsymbol{\phi}\dx\rightarrow0$ as $t_1\rightarrow t_0$ for arbitrary $\boldsymbol{\phi}\in W^{1,p^+}(\Omega)^d$. Note that
\begin{align*}
\int_{\Omega}\left(\uu(t_1)-\uu(t_0)\right)\cdot\boldsymbol{\phi}\dx
&=\int^{t_1}_{t_0}\langle\partial_t\uu,\boldsymbol{\phi}\rangle\dt\\
&=\int^{t_1}_{t_0}\int_{\Omega}\left((\uu\otimes\uu):\nabla\boldsymbol{\phi}-\bar{\SSS}:\nabla\boldsymbol{\phi}+\boldsymbol{f}\cdot \boldsymbol{\phi}\right)\dx\dt\\
&\leq\int^{t_1}_{t_0}\left(\|\uu\|^2_{2(p^-)'}+\|\bar{\SSS}\|_{(p^+)'}+\|\boldsymbol{f}\|_2\right)\|\boldsymbol{\phi}\|_{1,p^+}\dt\\
&\leq\left(\|\uu\|^2_{L^{2(p^-)'}(Q_T)}+\|\bar{\SSS}\|_{L^{(p^+)'}(Q_T)}+\|\boldsymbol{f}\|_{L^2(Q_T)}\right)|t_1-t_2|^{\frac{1}{p^+}}\|\boldsymbol{\phi}\|_{1,p^+}.
\end{align*}
Since the last term goes to $0$ as $t_1\rightarrow t_0$, we deduce that $\uu\in C_{\rm{w}}([0,T];W^{-1,(p^+)'}(\Omega))$. From this claim together with the fact $\|\uu\|_{L^{\infty}(0,T;L^2(\Omega)^d)}\leq C$, we further deduce that
\begin{equation}\label{initialweakconti}
\uu\in C_{\rm{w}}([0,T];L^2(\Omega)^d).
\end{equation}
Now, for fixed $t_0\in[0,T]$, we have from \eqref{Gthird1} that
\begin{align*}
\|\uu(t)-\uu(t_0)\|^2_2
&=\|\uu(t)\|^2_2+\|\uu(t_0)\|^2_2-2\int_{\Omega}\uu(t)\cdot\uu(t_0)\dx\\
&=\|\uu(t)\|^2_2-\|\uu(t_0)\|^2_2-2\int_{\Omega}(\uu(t)-\uu(t_0))\cdot\uu(t_0)\dx\\
&=\int^t_{t_0}\int_{\Omega}(\boldsymbol{f}\cdot\uu-\bar{\SSS}:\DD\uu)\dx\dt-2\int_{\Omega}(\uu(t)-\uu(t_0))\cdot\uu(t_0)\dx.
\end{align*}
Since $\int_{\Omega}(\boldsymbol{f}\cdot\uu-\bar{\SSS}:\DD\uu)\dx$ is integrable on $[0,T]$ and $\uu(t_0)\in L^2(\Omega)^d$, we can conclude that
\[\|\uu(t)-\uu(t_0)\|^2_2\rightarrow0\,\,\,{\rm{as}}\,\,\,t\rightarrow t_0,\]
which means that
\begin{equation}\label{initialveloconti}
\uu\in C([0,T];L^2(\Omega)^d).
\end{equation}

We next verify the time regularity for the concentration. From \eqref{Gsecond2}, we have
\[|\langle\partial_t\,c^N,\varphi\rangle|=\bigg|\int_{\Omega}(\uu^Nc^N\cdot\nabla\varphi-\q^N_c\cdot\nabla\varphi)\dx\bigg|\leq(\|\uu^Nc^N\|_2+\|\q_c^N\|_2)\|\varphi\|_{1,2}\]
By using \eqref{maxp} and the fact $\frac{2d}{d+2}<\frac{d+2}{2}\leq p^-$, we obtain
\[\int^T_0\norm{\partial_tc^{N}}^2_{W^{-1,2}(\Omega)}\dt\leq C\int^T_0(\|\uu^{N}c^{N}\|^2_2\dt+\|\q_c^{N}\|^2_2)\dt\leq C\int^T_0\|\uu^{N}\|^2_2\dt+C\leq C\]
Therefore, we finally have
\[\|\partial_tc^N\|_{L^2(0,T;W^{-1,2}(\Omega))}\leq C,\]
which implies
\[\partial_t\,c\in L^2(0,T;W^{-1,2}(\Omega)).\]
We can now conclude that
\begin{equation}\label{initialconcconti}
c\in C([0,T];L^2(\Omega)).
\end{equation}

As a final step, to complete the proof, it remains to show that
\[\bar{\SSS}=\SSS(c,\DD\uu)\qquad{\rm{and}}\qquad\bar{\q}_c=\q_c(c,\nabla c,\DD\uu).\]
These equalities will be verified in the next subsection.
\end{subsection}

\begin{subsection}{Identification of the limits $\bar{\SSS}=\SSS(c,\DD\uu)$ and $\bar{\q}_c=\q_c(c,\nabla c,\DD\uu)$}
In this subsection, we shall identify the limit functions to complete the proof. We first note that from \eqref{Gsecond1}, we have
\begin{align*}
\int_{Q_T}\boldsymbol{f}\cdot\uu^N\dx\dt
&=\int_{Q_T}\left(\partial_t\uu^N\cdot\uu^N-(\uu^N\otimes\uu^N):\DD\uu^N+\SSS^N:\DD\uu^N\right)\dx\dt\\
&=\frac{1}{2}\int_{\Omega}|\uu^N(T)|^2\dx-\frac{1}{2}\int_{\Omega}|\uu^N(0)|^2\dx+\int_{Q_T}\SSS^N:\DD\uu^N\dx\dt,
\end{align*}
where we have used that ${\rm{div}}\,\uu^N=0$. Also, by \eqref{timeest2}, we can test by $\uu$ in \eqref{Gthird1}. Therefore, as ${\rm{div}}\,\uu=0$, we have that
\begin{align*}
\int_{Q_T}\boldsymbol{f}\cdot\uu\dx\dt
&=\int^T_0\langle\partial_t\uu,\uu\rangle\dt-\int_{Q_T}\left((\uu\otimes\uu):\DD\uu-\bar{\SSS}:\DD\uu\right)\dx\dt\\
&=\frac{1}{2}\int_{\Omega}|\uu(T)|^2\dx-\frac{1}{2}\int_{\Omega}|\uu(0)|^2\dx+\int_{Q_T}\bar{\SSS}:\DD\uu\dx\dt,
\end{align*}
Therefore, it is straightforward to see that
\begin{align*}
&\hspace{-15mm}\int_{Q_T}\SSS^N:\DD\uu^N\dx\dt-\int_{Q_T}\bar{\SSS}:\DD\uu\dx\dt\\
&=\int_{Q_T}\boldsymbol{f}\cdot\uu^N\dx\dt-\int_{Q_T}\boldsymbol{f}\cdot\uu\dx\dt\\
&\hspace{5mm}-\frac{1}{2}\int_{\Omega}|\uu^N(T)-\uu(T)|^2\dx-\int_{\Omega}\uu^N(T)\cdot\uu(T)\dx\\
&\hspace{5mm}+\int_{\Omega}|\uu(T)|^2\dx+\frac{1}{2}\int_{\Omega}|P^N\uu_0|^2\dx-\frac{1}{2}\int_{\Omega}|\uu_0|^2\dx.
\end{align*}
It is obvious that $\int_{Q_T}\boldsymbol{f}\cdot\uu^N\dx\dt\rightarrow\int_{Q_T}\boldsymbol{f}\cdot\uu\dx\dt$ and $\|P^N\uu_0-\uu_0\|_2\rightarrow0$. Also, \eqref{initialweakconti} implies that $\uu^N(T)\rightharpoonup\uu(T)$ weakly in $L^2(\Omega)$.
Thus we have
\begin{equation}\label{importconv}
\limsup_{N\rightarrow\infty}\int_{Q_T}\SSS^N:\DD\uu^N\dx\dt\leq\int_{Q_T}\bar{\SSS}:\DD\uu\dx\dt.
\end{equation}

Next, by \eqref{seconv5} and the dominated convergence theorem, we obtain for arbitrary but fixed $\boldsymbol{\phi}\in C^{\infty}([0,T];C^{\infty}_0(\Omega)^d)$ that
\begin{equation}\label{Lqconvforallq}
\SSS(c^N,\DD\boldsymbol{\phi})\rightarrow\SSS(c,\DD\boldsymbol{\phi})\,\,\,{\rm{strongly}}\,\,\,{\rm{in}}\,\,\,L^q(Q_T)^{d\times d}\,\,\,{\rm{for}}\,\,\,{\rm{all}}\,\,\,q\in(1,\infty).
\end{equation}

Then, by the monotonicity assumption \eqref{S2}, we deduce for any $\boldsymbol{\phi}\in C^{\infty}([0,T];C^{\infty}_0(\Omega)^d)$ that
\begin{align*}
0
&\leq\limsup_{N\rightarrow\infty}\int_{Q_T}\left(\SSS(c^N,\DD\uu^N)-\SSS(c^N,\DD\boldsymbol{\phi})\right):(\DD\uu^N-\DD\boldsymbol{\phi})\dx\dt\\
&=\limsup_{N\rightarrow\infty}\int_{Q_T}(\SSS(c^N,\DD\uu^N):\DD\uu^N-\SSS(c^N,\DD\uu^N):\DD\boldsymbol{\phi}\\
&\hspace{5mm}-\SSS(c^N,\DD\boldsymbol{\phi}):\DD\uu^N+\SSS(c^N,\DD\boldsymbol{\phi}):\DD\boldsymbol{\phi})\dx\dt\\
&\leq\int_{Q_T}\left(\bar{\SSS}:\DD\uu-\bar{\SSS}:\DD\boldsymbol{\phi}-\SSS(c,\DD\boldsymbol{\phi}):\DD\uu+\SSS(c,\DD\boldsymbol{\phi}):\DD\boldsymbol{\phi}\right)\dx\dt\\
&=\int_{Q_T}\left(\bar{\SSS}-\SSS(c,\DD\boldsymbol{\phi})\right):(\DD\uu-\DD\boldsymbol{\phi})\dx\dt.
\end{align*}
Now we use Minty's trick. Since $C^{\infty}([0,T];C^{\infty}_0(\Omega)^d)$ is dense in $W_{p(c)}(Q_T)$, we can set $\boldsymbol{\phi}=\uu\pm\lambda\vv$ with $\lambda>0$ and $\vv\in C^{\infty}([0,T];C^{\infty}_0(\Omega)^d)$. By passing to the limit $\lambda\rightarrow0$ and using the continuity of $\SSS$, we deduce the desired identification $\bar{\SSS}=\SSS(c,\DD\uu)$.

Next, we shall show the almost everywhere convergence of $\DD\uu^N$ to $\DD\uu$. Let $\boldsymbol{B}\in\R^{d\times d}$ be a bounded symmetric function. Then with the same procedure as above and the identification $\bar{\SSS}=\SSS(c,\DD\uu)$, we have that
\begin{align*}
0
&\leq\limsup_{N\rightarrow\infty}\int_{Q_T}\left(\SSS(c^N,\DD\uu^N)-\SSS(c^N,\boldsymbol{B})\right):(\DD\uu^N-\boldsymbol{B})\dx\dt\\
&\leq\int_{Q_T}\left(\SSS(c,\DD\uu)-\SSS(c,\boldsymbol{B})\right):(\DD\uu-\boldsymbol{B})\dx\dt.
\end{align*}
We set $\boldsymbol{B}=\DD\uu\chi_{\{|\DD\uu|<\lambda\}}$ for $\lambda>0$ and let $0<k<\lambda$. Then we have
\begin{align*}
&\hspace{-15mm}\limsup_{N\rightarrow\infty}\int_{\{|\DD\uu|<k\}}\left(\SSS(c^N,\DD\uu^N)-\SSS(c^N,\DD\uu)\right):(\DD\uu^N-\DD\uu)\dx\dt\\
&=\limsup_{N\rightarrow\infty}\int_{\{|\DD\uu|<k\}}\left(\SSS(c^N,\DD\uu^N)-\SSS(c^N,\boldsymbol{B})\right):(\DD\uu^N-\boldsymbol{B})\dx\dt\\
&\leq\limsup_{N\rightarrow\infty}\int_{Q_T}\left(\SSS(c^N,\DD\uu^N)-\SSS(c^N,\boldsymbol{B})\right):(\DD\uu^N-\boldsymbol{B})\dx\dt\\
&=\int_{Q_T}\left(\SSS(c,\DD\uu)-\SSS(c,\boldsymbol{B})\right):(\DD\uu-\boldsymbol{B})\dx\dt\\
&=\int_{\{|\DD\uu|>\lambda\}}\left(\SSS(c,\DD\uu)-\SSS(c,\boldsymbol{B})\right):(\DD\uu-\boldsymbol{B})\dx\dt.
\end{align*}
Since the first integral is independent of $\lambda>0$, passing $\lambda$ to $\infty$ gives us that
\[\limsup_{N\rightarrow\infty}\int_{\{|\DD\uu|<k\}}\left(\SSS(c^N,\DD\uu^N)-\SSS(c^N,\DD\uu)\right):(\DD\uu^N-\DD\uu)\dx\dt=0.\]
Since $c^N$ converges to $c$ uniformly and $\SSS$ is strictly monotone, the only way the above can be true is that $\DD\uu^N$ tends to $\DD\uu$ almost everywhere in the set $\{|\DD\uu|< k\}$. If we pass $k$ to $\infty$, we finally have
\begin{equation}\label{Dua.e.conv}
\DD\uu^N\rightarrow\DD\uu\,\,\,{\rm{a.e.}}\,\,{\rm{in}}\,\,\,Q_T.
\end{equation}
Now by \eqref{seconv5}, \eqref{Dua.e.conv} and the dominated convergence theorem, 
\[\boldsymbol{K}(c^N,|\DD\uu^N|)\rightarrow \boldsymbol{K}(c,|\DD\uu|)\,\,\,{\rm{strongly}}\,\,{\rm{in}}\,\,L^q(Q_T)\,\,\,{\rm{for}}\,\,\,{\rm{all}}\,\,\,q\in(1,\infty).\]
Finally, together with \eqref{seconv4}, we obtain that
\[\q^N_c\rightharpoonup\q_c(c,\nabla c,\DD\uu)\,\,\,{\rm{weakly}}\,\,{\rm{in}}\,\,L^2(Q_T)^d.\]
Therefore, by the uniqueness of the weak limit, we have 
\[\bar{\q}_c=\q_c(c,\nabla c,\DD\uu),\]
which completes the proof of Theorem \ref{mainthm}.
\end{subsection}
\end{section}

\begin{section}{Conclusions}
We have established the existence of global weak solutions to the system of PDEs consisting of the generalized Navier--Stokes equations with a power-law type stess-strain relation where the power-law index depends on the concentration, coupled to a convection-diffusion equation. Our main technical tools included generalized monotone operator theory and parabolic De Giorgi--Nash--Moser regularity theory from which we obtained the H\"older continuity of the approximate concentrations. Thus the restriction $p^->\frac{d+2}{2}$ was essential in our analysis since $p^->\frac{d+2}{2}$ was the minimum requirement to ensure the H\"older continuity of the concentration, which then guaranteed the validity of necessary function space theory in variable-exponent spaces.

However, as was mentioned in the introduction, we can apply our method to the electro-rheological fluid flow problem which is of independent interest. In this case, we can prove the existence of weak solutions with lower values of $p^-$. More precisely, we have the following result as a corollary of Theorem \ref{mainthm}.
\begin{corollary}
Let $\Omega\subset\R^d$ with $d\geq2$ be a bounded open Lipschitz domain, and assume that $\boldsymbol{f}\in L^{(p^+)'}(Q_T)^d$ and $\uu_0\in L^2(\Omega)^d$.  If $p$ is a H\"older continuous function with $p^->\frac{3d+2}{d+2}$, then there exists a weak solution $\uu\in C([0,T];L^2(\Omega)^d)\cap L^{p^-}(0,T;W^{1,p^-}_{0,\,{\rm{div}}}(\Omega)^d)\cap W_{p(\cdot)}(Q_T)$ satisfying
\[\int_{Q_T}\SSS(\DD\uu):\boldsymbol{D\psi}\dx\dt
=\int_{Q_T}\left(\boldsymbol{f}\cdot\boldsymbol{\psi}+(\uu\otimes\uu):\boldsymbol{D\psi}+\uu\cdot\partial_t\boldsymbol{\psi}\right)\dx\dt+\int_{\Omega}\uu_0(x)\cdot\boldsymbol{\psi}(x,0)\dx\]
for all $\boldsymbol{\psi}\in C^{\infty}_{0,{\rm{div}}}(\Omega\times [0,T))^d$ where
\[\SSS(\DD\uu)=(C_1+C_2|\DD\uu|^2)^{\frac{p(x,t)-2}{2}}\DD\uu.\]
\end{corollary}
Unlike the model used in this paper, such an electro-rheological fluid flow problem is free from a coupled concentration equation, so there is a possibility to improve the above result. For example, if one could develop parabolic $L^{\infty}$ or Lipschitz truncation techniques for variable-exponent spaces, we could perhaps obtain the compactness of the convective term with lower values of $p^-$, and hence we could prove the existence of weak solutions with a weaker assumption on $p^-$. Therefore, an interesting direction for future research is to generalize the solenoidal parabolic Lipschitz truncation technique introduced in \cite{BDS2013} to the case of variable-exponent spaces, so that we can prove the existence of weak solutions for $p^->\frac{2d}{d+2}$.
\end{section}

\section*{Acknowledgements}
Seungchan Ko's work was supported by the UK Engineering and Physical Sciences Research Council [EP/L015811/1].

\bibliography{new_references}
\bibliographystyle{abbrv}


\end{document}